\documentclass{amsart}
\usepackage{amsmath}
\usepackage{hyperref}
\usepackage{mathtools}
\usepackage{tikz-cd}
\usepackage[all]{xy}
\SelectTips{cm}{12}
\usepackage{url}
\usepackage{amssymb}
\makeatletter
\let\origsection\section
\renewcommand\section{\@ifstar{\starsection}{\nostarsection}}
\newcommand\nostarsection[1]
{\sectionprelude\origsection{#1}\sectionpostlude}
\newcommand\starsection[1]
{\sectionprelude\origsection*{#1}\sectionpostlude}
\newcommand\sectionprelude{%
  \vspace{1em}
}
\newcommand\sectionpostlude{}
\makeatother

\newcommand{\poin}{\mathcal{P}}
\newcommand{\xys}{S}
\newcommand{\XX}{\mathcal{X}}
\newcommand{\YY}{\mathcal{Y}}
\newcommand{\ZZ}{\mathcal{Z}}
\DeclareMathOperator{\im}{im}
\DeclareMathOperator{\id}{id}
\DeclareMathOperator{\Sym}{\mathrm{Sym}}

\newtheorem{proposition}{Proposition}[section]
\newtheorem{thm}[proposition]{Theorem}
\newtheorem{lemma}[proposition]{Lemma}
\newtheorem{corollary}[proposition]{Corollary}
\theoremstyle{definition}
\newtheorem{defi}[proposition]{Definition}

\theoremstyle{remark}
\newtheorem*{remark}{Remark}
\title{Betti numbers of configuration spaces of surfaces}
\author{Gabriel C. Drummond-Cole}
\thanks{The first author was supported by IBS-R003-D1}
\address{Center for Geometry and Physics, Institute for Basic Science (IBS), Pohang 37673, Republic of Korea}
\author{Ben Knudsen}
\thanks{The second author was supported by NSF award 1606422}
\address{Department of Mathematics, Harvard University, Cambridge 02138, USA}

\begin{document}
\begin{abstract}
We give explicit formulas for the Betti numbers, both stable and unstable, of the unordered configuration spaces of an arbitrary surface of finite type.
\end{abstract}
\maketitle
\section{Introduction}
In this paper, we study the configuration space of $k$ unordered points in a surface $\Sigma$ of finite type, which is defined as the quotient
\[B_k(\Sigma) =\{(x_1,\ldots, x_k)\in \Sigma^k | x_i\ne x_j \text{ if } i\neq j\}/S_k,\] where $S_k$ denotes the symmetric group on $k$ letters. Our goal is the explicit computation of the Betti number $\beta_i(B_k(\Sigma))$ for every choice of $i$, $k$, and $\Sigma$. This computation is divided among the following results according to the nature of $\Sigma$:\vspace{.1in}

\begin{center} 
{\def\arraystretch{1.2}
\begin{tabular}{  c  c  }  
Type of surface & Computation
\\
\hline Closed nonorientable & Proposition \ref{prop:closed nonorientable} \\ 
Open nonorientable & Proposition \ref{prop: open nonorientable} \\ 
Open orientable & Proposition \ref{prop: punctured surface Betti numbers} \\ 
Closed orientable & Corollaries \ref{cor:unstable k+1}, \ref{cor: second unstable betti number}, and \ref{cor: stable numbers}\\
\end{tabular} 
}\end{center}\vspace{.1in}

Some of these computations are known. The case of the plane is treated by Arnold in \cite{Arnold:CRCBG}, the case of a once-punctured orientable surface by B\"{o}digheimer--Cohen in \cite{BodigheimerCohen:RCCSS}, the case of the sphere by Salvatore in \cite{Salvatore:CSSHLS} (see also \cite{RandalWilliams:TCHCSS}), the case of the real projective plane by Wang in \cite{Wang:OBGRP2}, and the case of a general closed nonorientable surface by the second author in \cite{Knudsen:BNSCSVFH}. In addition, during the writing of this work, the authors learned of two independent computations in the case of the torus, due to Maguire \cite{Maguire:CCCS} and Schiessl \cite{Schiessl:BNUCST}, respectively. We include all of these computations below for the sake of completeness.

The principal contribution of this work is the case of a general closed surface $\Sigma_g$ of genus $g$. In this most difficult case, the resulting formulas are rather complicated, and we complement them with two asymptotic results characterizing the behavior of $\beta_i(B_k(\Sigma_g))$ viewed alternately as a function of $i$ or of $g$. These results appear below as Corollaries \ref{cor:genus stability} and \ref{cor: fixed g polynomials}. According to unpublished work of Dan Petersen, the former result, which one might characterize as a kind of ``homological stability in genus,'' has a conceptual explanation in terms of representation stability for the symplectic groups.

In order to describe our approach, we recall a general method for understanding the rational homology of configuration spaces of manifolds. To a manifold $M$, there corresponds a graded Lie algebra $\mathfrak{g}_M$ built from the cohomology of $M$, and the graded vector space $H_*(B_k(M);\mathbb{Q})$ coincides with a particular summand of the Lie algebra homology of $\mathfrak{g}_M$. This Lie algebra homology may be computed by means of the \emph{Chevalley--Eilenberg complex} (see Definition \ref{defi: CE complex}), and the work of this paper consists in the application of algebraic and combinatorial techniques to this complex.

The Chevalley--Eilenberg complex has been a ubiquitous presence in the study of the unordered configuration spaces of manifolds (as well as the ordered---see \cite{Getzler:RMHMOCS}). Prominent examples of its appearance include the work of B\"{o}digheimer--Cohen--Taylor \cite{BodigheimerCohenTaylor:OHCS}, B\"{o}digheimer--Cohen \cite{BodigheimerCohen:RCCSS}, and F\'{e}lix--Thomas \cite{FelixThomas:RBNCS}, building on McDuff's foundational work \cite{McDuff:CSPNP}; and the work of F\'{e}lix--Tanr\'{e} \cite{FelixTanre:CAUCS} following Totaro \cite{Totaro:CSAV}. In order to make the identification with Lie algebra homology, each of these works requires assumptions about the background manifold---orientability, for example. More recently, the identification was established in full generality by the second author in \cite{Knudsen:BNSCSVFH} using the theory of \emph{factorization homology} developed by Ayala--Francis \cite{AyalaFrancis:FHTM} and Lurie \cite[Ch. 5]{Lurie:HA}, among others.

Although the Chevalley--Eilenberg complex has appeared widely in the past, it has often been the practice to deal with it one subcomplex at a time---to treat separately the computations for the configuration spaces of five and of twenty points, say. Our approach is to treat the complex as a whole, performing the computation for $B_k(\Sigma)$ simultaneously for all $k$. It is this simultaneity that renders the computations feasible in practice.

\subsection*{Acknowledgements}
The first author thanks Christoph Schiessl for helpful conversation. The second author thanks Jordan Ellenberg, Benson Farb, and Joel Specter. 

\section{Recollections}

\subsection{General conventions}\label{subsec: conventions}

We work throughout with graded vector spaces or (co)chain complexes over the ground ring $\mathbb{Q}$. Degree is understood homologically; that is, the differential of a chain complex decreases degree while that of a cochain complex increases degree. In addition, our chain complexes will often carry an auxiliary grading, called \emph{weight}. Degree is generically indexed by $i$ and weight by $k$. We denote the dimension of the degree $i$ summand of a graded vector space $V$ by $\dim_i$. If $V$ is weighted, we denote the dimension of the degree $i$ and weight $k$ summand by $\dim_{i,k}V$. For $X$ a topological space, the Betti number $\beta_i(X)$ is $\dim_i H(X;\mathbb{Q})$.

The $n$th \emph{suspension} of the graded vector space $V$ is the graded vector space $V[n]$ with $V[n]_i=V_{i-n}$, and the element of $V[n]$ corresponding to $x\in V$ is denoted $\sigma^nx$. Vector spaces are identified with graded vector spaces concentrated in degree 0; for example, \[\mathbb{Q}[n]_i=\begin{cases}
\mathbb{Q}&\quad i=n\\
0&\quad \text{else.}
\end{cases}\] The degree of a homogeneous element $x$ is written $|x|$.

Graded vector spaces will typically be finite dimensional in each degree. In this situation, it is convenient to collect the dimensions of the summands of $V$ into its \emph{Poincar\'{e} series} \[
\poin_V(t)=\sum_{n\in\mathbb{Z}} \dim_i(V)t^i.
\] For example, we have the equalities \begin{align*}\poin_{\Lambda[x]}(t)&=1+t^{|x|}\\
\poin_{\mathbb{Q}[x]}(t)&=\sum_{n\geq0}t^{n|x|}=\frac{1}{1-t^{|x|}}, \quad |x|\neq0.
\end{align*} Poincar\'{e} series are additive under direct sum and, modulo convergence issues for unbounded complexes, multiplicative under tensor product. Moreover, for a complex $(V,d)$, we have the equality
\[
\poin_{H(V,d)}(t) = \poin_{\ker(d)}(t) + t^{|d|}(\poin_{\ker(d)}(t) - \poin_V(t)).\]

The symmetric algebra $\Sym(V)$ is understood in the graded sense; that is \[\Sym(V):=\mathbb{Q}[V_{\mathrm{even}}]\otimes \Lambda[V_{\mathrm{odd}}].\] If $V$ is weighted, then $\Sym(V)$ inherits a weight grading by specifying that the inclusion of $V$ be weight preserving and that weights add under multiplication. Note that it is the homological degree alone that determines whether a homogeneous element is a polynomial or an exterior generator. 

We will typically work with $\Sym(V)$ where $V$ is a weighted complex concentrated in weights 1 and 2. In these situations, we employ the convention that a variable without a tilde has weight one, while a variable with a tilde has weight two.

\subsection{Configuration spaces and Lie algebras}

Recall that a graded Lie algebra is a graded vector space $\mathfrak{g}$ equipped with a linear map $[-,-]:\mathfrak{g}^{\otimes2}\to \mathfrak{g}$ satisfying the graded antisymmetry and graded Jacobi identities.

\begin{defi}\label{defi: CE complex}
Let $\mathfrak{g}$ be a graded Lie algebra. The \emph{Chevalley--Eilenberg complex} of $\mathfrak{g}$ is the chain complex $CE(\mathfrak{g}[1])$ whose underlying graded vector space is $\Sym(\mathfrak{g}[1])$, which carries the structure of the cofree conilpotent cocommutative coalgebra on the graded vector space $\mathfrak{g}[1]$, and whose differential is the unique coderivation $D$ of that coalgebra structure such that \[D(\sigma x\cdot \sigma y)=(-1)^{|x|}\sigma[x,y].\]
\end{defi}

\begin{remark}
The complex $CE(\mathfrak{g})$ computes the Lie algebra homology of $\mathfrak{g}$; the reader interested in more may consult any of many expositions, for instance~\cite[XIII]{CartanEilenberg:HA}, \cite[7]{Weibel:IHA}, or \cite[22]{FelixHalperinThomas:RHT}. We note for the sake of completeness that the differential is given explicitly by the formula \[D(\sigma x_1\cdots \sigma x_n)=\sum_{1\leq i<j\leq n}(-1)^{|x_i|}\epsilon(x_1,\ldots, x_n)\sigma[x_i,x_j]\cdot\sigma x_1\cdots \widehat{\sigma x_i}\cdots \widehat{\sigma x_k}\cdots \sigma x_n,\] where the sign $\epsilon$ is determined by \[\sigma x_1\cdots \sigma x_n=\epsilon(x_1,\dots, x_n)\sigma x_i\cdot\sigma x_j\cdot\sigma x_1\cdots \widehat{\sigma x_i}\cdots \widehat{\sigma x_j}\cdots \sigma x_n.\] We shall not make use of this formula, as a more convenient sign convention, made explicit at the close of the section, is available in the case of interest.
\end{remark}

Now, let $\Sigma$ be a surface of finite type. We write $\mathbb{Q}^w$ for the orientation sheaf of $\Sigma$ and recall that there is an isomorphism $(\mathbb{Q}^w)^{\otimes 2}\cong\mathbb{Q}$ inducing a cup product from twisted to ordinary cohomology. We write $\mathfrak{g}_\Sigma$ for the graded Lie algebra given additively by
\[\mathfrak{g}_\Sigma=H_c^{-*}(\Sigma;\mathbb{Q}^w)[1]\oplus H_c^{-*}(\Sigma;\mathbb{Q})[2],\] where $H_c$ denotes compactly supported cohomology, with the nonzero components of the bracket determined by the cup product according to the equation \[[\sigma\alpha,\sigma\beta]=(-1)^{|\beta|}\sigma^2(\alpha\smile\beta).\]

\begin{remark} If $\Sigma$ is orientable, then $\mathfrak{g}_\Sigma$ is the tensor product of the cohomology of $\Sigma$ and the free graded Lie algebra on a single generator of degree 1, equipped with the canonical Lie algebra structure on the tensor product of a commutative algebra and a Lie algebra. If $\Sigma$ is nonorientable, there is an analogous characterization as the the super tensor product of a commutative superalgebra with a Lie superalgebra.
\end{remark}

The filtration of $\mathfrak{g}_\Sigma$ by bracket length is canonically split, and we regard $\mathfrak{g}_\Sigma$ and thereby $CE(\mathfrak{g}_\Sigma)$ as weight graded according to the induced grading. We have the following result, which we do not state in the greatest possible generality.

\begin{thm}\label{thm:Knudsen}\cite{Knudsen:BNSCSVFH}
There is an equality $$\beta_i(B_k(\Sigma))= \dim_{i,k}H(CE(\mathfrak{g}_\Sigma)).$$
\end{thm}

We close this section with a remark on signs. From the definitions, we have \[CE(\mathfrak{g}_\Sigma)=\Sym\big(H_c^{-*}(\Sigma;\mathbb{Q}^w)[2]\oplus H_c^{-*}(\Sigma;\mathbb{Q})[3]\big),\] and $D$ is determined as a coderivation by specifying that 
\[D(\sigma^2\alpha\cdot \sigma^2\beta)=(-1)^{|\alpha|+|\beta|+1}\sigma^3(\alpha\smile\beta).\] All of the calculations below are performed with a differential, also called $D$, that omits the sign $(-1)^{|\alpha|+|\beta|+1}$. This omission is justified by a linear change of variables; indeed, multiplication by $-1$ in the even part of $H_c^{-*}(\Sigma;\mathbb{Q})[3]$ eliminates the sign.

\section{Nonorientable and open surfaces}\label{sec: non-orientable and open surfaces}

In this section, we compute the Betti numbers of $B_k(\Sigma)$, where $\Sigma$ is either nonorientable or both orientable and open. The bulk of the former computation was carried out in \cite{Knudsen:BNSCSVFH} and the bulk of the latter in \cite{BodigheimerCohen:RCCSS}; nevertheless, we include them here for the sake of completeness and as a warmup for the more involved computations to come.

\subsection{Nonorientable surfaces} Let $N_{h}\cong\#_h\mathbb{RP}^2$ denote the nonorientable surface of Euler characteristic $2-h$, and let $N_{h,n}=N_h\setminus S$, where $|S|=n$. Using Poincar\'{e} duality and elementary algebraic topology, one finds that \[
H_c^{-*}(N_{h,n};\mathbb{Q}^w)\cong \mathbb{Q}[-1]^{h+n-1}\oplus \mathbb{Q}[-2]\] and
\[H_c^{-*}(N_{h,n};\mathbb{Q})\cong \begin{cases} \mathbb{Q}\oplus \mathbb{Q}[-1]^{h-1}&\quad n=0\\
\mathbb{Q}[-1]^{h+n-2}&\quad\text{else.}
\end{cases}
\] Thus, the twisted cup product $H_c^{-*}(N_{h,n};\mathbb{Q}^w)^{\otimes 2}\to H_c^{-*}(N_{h,n};\mathbb{Q})$ vanishes for degree reasons, and we conclude that the Chevalley--Eilenberg differential vanishes.

In the closed case $n=0$, we have the following:

\begin{proposition}\label{prop:closed nonorientable}
For any $k\geq0$, \[
\beta_i(B_k(N_{h}))
=\left\{
	\begin{array}{ll}
	\binom{h+i-2}{h-2} + \binom{h+i-5}{h-2}&i\leq k\\
	\binom{h+i-5}{h-2}& i=k+1\\
	0&\text{else}.
	\end{array}
\right.
\]
where $\binom{-1}{-1}\coloneqq 1$.
\end{proposition}
\begin{proof}
Since $\mathfrak{g}_{N_h}$ is Abelian, there is no differential in $CE(\mathfrak{g}_{N_h})$, so the $i$th Betti number of $B_k(N_{h})$ is given by the dimension of the summand of degree $i$ and weight $k$ in
\[
CE(\mathfrak{g}_{N_h})\cong \mathbb{Q}[p,\tilde{u}_1,\ldots, \tilde{u}_{h-1}]\otimes \Lambda[\tilde{v},u_1,\ldots, u_{h-1}]
\]
where $|p|=0$, $|u_j|=1$, $|\tilde{u}_j|=2$, and $|\tilde{v}|=3$ (this is a slight abuse of notation as the $u_j$ classes are from twisted cohomology groups of $N_{h}$ while the $\tilde{u}_j$ classes are from untwisted cohomology groups of $N_{h}$). Decomposing this expression using the elements $\tilde u_{h-1}^ru_{h-1}^s$ for $r\geq0$ and $s\in \{0,1\}$ yields the recursive formula $$\beta_i(B_k(N_h))= \sum_{j=0}^i\dim H_{i-j}(B_{k-j}(N_{h-1});\mathbb{Q})$$ for $h\geq2$, together with the base case $$\beta_i(B_k(N_1))=\begin{cases}
1&\quad i\in\{0,3\}\\
0&\quad\text{else}
\end{cases}$$ for $k>1$. An easy induction on $h$ using the identity
\[
\sum_{j=0}^i \binom{n+j}{n}=\binom{n+i+1}{n+1}
\]
(valid for $n\ge -1$) completes the proof.
\end{proof}

In the punctured case, the calculation is even simpler:

\begin{proposition}\label{prop: open nonorientable}
For any $k\geq0$ and $n\geq1$, \[\beta_i(B_k(N_{h,n}))=\left\{
	\begin{array}{ll}
	\binom{h+n+i-3}{h+n-3} + \binom{h+n+i-4}{h+n-3}&i\leq k\\
	0&\text{else}.
	\end{array}
\right.\]
\end{proposition}
\begin{proof}
As before, there is no differential in the Chevalley--Eilenberg complex, so the $i$th Betti number of $B_k(N_{h,n})$ is given by the dimension of the summand of degree $i$ and weight $k$ in $$CE(\mathfrak{g}_{h,n})\cong\mathbb{Q}[p, \tilde u_1,\ldots, \tilde u_{h+n-2}]\otimes\Lambda[u_1,\ldots, u_{h+n-1}],$$ where $|p|=0$, $|u_j|=1$, and $|\tilde u_j|=2$ (the same caveat about naming applies), and we again have the recursive formula $$\beta_i(B_k(N_{h,n}))= \sum_{j=0}^k\dim H_{i-j}(B_{k-j}(N_{h-1,n});\mathbb{Q})$$ for $h\geq2$, together with the base case $$\beta_i(B_k(N_{1,1}))=\begin{cases}
1&\quad i\in\{0,1\}\\
0&\quad\text{else}
\end{cases}$$ for $k\geq1$. Since $\beta_i(B_k(N_{h,n}))=\beta_i(B_k(N_{h+1,n-1}))$ for $n\geq2$, it suffices to verify the case $n=1$, which follows as before by induction on $h$.
\end{proof}

\subsection{Open orientable surfaces} We turn our attention to $\Sigma_{g,n}$, the $n$-punctured orientable surface of genus $g$, with $n>0$. As a graded vector space, the Lie algebra $\mathfrak{g}_{\Sigma_{g,n}}$ is given by
\[
H_c^{-*}(\Sigma_{g,n};\mathbb{Q})[1]
\oplus
H_c^{-*}(\Sigma_{g,n};\mathbb{Q})[2].
\] One knows that
\[
H_c^{-*}(\Sigma_{g,n});\mathbb{Q})\cong \mathbb{Q}[-1]^{2g+n-1}\oplus \mathbb{Q}[-2],
\] and the only nonzero brackets take the paired classes in the lefthand copy of $H_c^1(\Sigma_{g,n})$ to their cup product in the righthand copy of $H^2_c(\Sigma_{g,n})$. Denoting by $\mathfrak{h}$ the Lie subalgebra spanned by these $2g+1$ classes, we may decompose $\mathfrak{g}_{\Sigma_{g,n}}$ as a sum of $\mathfrak{h}$ and an Abelian Lie algebra spanned by the remaining classes. This sum passes to a tensor product at the level of Chevalley--Eilenberg complexes, and we obtain the decomposition
\[
CE(\mathfrak{g}_{\Sigma_{g,n}})\cong
CE(\mathfrak{h})\otimes W_{g,n}.
\]
Here $W_{g,n}=\Sym(p, \tilde{a}_r,\tilde{b}_r, u_s,\tilde u_s)$, $|p|=0$, $|u_s|=1$, $|\tilde u_s|=2$, $|\tilde{a}_r|=|\tilde b_r|=2$, and $(1,1)\leq (r,s)\leq (g,n-1)$.

The Lie algebra homology of $\mathfrak{h}$ was computed in dual form in~\cite{BodigheimerCohen:RCCSS}.

\begin{thm}\label{thm:BodigheimerCohen}\cite[Theorem D.]{BodigheimerCohen:RCCSS}
$$
\dim_{i,k}(
H(CE(\mathfrak{h}))=\begin{cases}
\binom{2g}{i}-\binom{2g}{i-2}&\quad 0\leq i\leq g,\, i=k\\
\binom{2g}{i-1}-\binom{2g}{i+1}&\quad g+1\leq i\leq 2g+1,\, i=k-1\\
0&\quad \text{else.}
\end{cases}$$ 
\end{thm}

As for $W_{g,n}$, we have the following.

\begin{lemma}\label{lemma: punctured_oriented_lemma}
For any $g,k\geq0$ and $n\geq1$, 
\[
\dim_{i,k}(W_{g,n})=\begin{cases}\sum_{l=0}^{\lfloor\frac{i}{2}\rfloor}\binom{n+i-2l-2}{n-2}\binom{2g+l-1}{2g-1}&\quad i\leq k\\
0&\quad\text{else.}
\end{cases}
\]
\end{lemma}
\begin{proof}
Write $W_*=\mathbb{Q}[\tilde a_r, \tilde b_r]$ and $W_+=\mathbb{Q}[\tilde{u}_s]\otimes \Lambda[u_s]$. Then for $i\le k$, \begin{align*}
\dim_{i,k}(W_{g,n})&=\dim_{i,k}(\mathbb{Q}[p]\otimes W_*\otimes W_+)\\
&=\sum_{j=0}^k\dim_{i,k-j}(W_*\otimes W_+)\\
&=\dim_{i}(W_*\otimes W_+),
\end{align*} since $\dim_{i,k}(W_*\otimes W_+)=0$ whenever $i\neq k$. We also have \[\dim_{i}(W_*)=\begin{cases}
\binom{2g+l-1}{2g-1}&\quad i=2l\\
0&\quad\text{else,}
\end{cases}\] since the displayed binomial coefficient is the number of ways of writing $l$ as the sum of $2g$ nonnegative integers, while 
\[
\dim_{i}(W_+)=
\binom{n+i-2}{n-2}\] 
since the displayed binomial coefficient is the number of ways of writing $i$ as the sum of $n-1$ nonnegative integers. With these calculations in hand, the claim follows from the equality
\[\dim_{i}(W_*\otimes W_+) = \sum_{2l+m=i}\dim_{2l}(W_*)\dim_{m}(W_+).
\]
\end{proof}

Combining these calculations, we deduce the following:

\begin{proposition}\label{prop: punctured surface Betti numbers} For $g,k\geq0$ and $n\geq1$, the Betti number $\beta_i(B_k(\Sigma_{g,n}))$ is:
{\small\[
\beta_i(B_k(\Sigma_{g,n}))=\begin{cases}
 {\displaystyle\sum_{j=0}^g}\left[\binom{2g}{j}-\binom{2g}{j-2}\right]{\displaystyle\sum_{l=0}^{\lfloor\frac{i-j}{2}\rfloor}}\binom{2g+l-1}{2g-1}\bigg[\\\qquad{}\qquad{}\binom{n+i-j-2l-2}{n-2}+\binom{n+i+j-2g-2l-3}{n-2}\bigg]
&i\le k-1\\
{\displaystyle\sum_{j=0}^g}\left[\binom{2g}{j}-\binom{2g}{j-2}\right]{\displaystyle\sum_{l=0}^{\lfloor\frac{i-j}{2}\rfloor}}\binom{2g+l-1}{2g-1}\binom{n+i-j-2l-2}{n-2}&i=k\\
\displaystyle 0&\text{else.}
\end{cases}
 \]
}
\end{proposition}

\begin{proof}
For $i\leq k-1$, we have 
\begin{multline*}
\beta_i(B_k(\Sigma_{g,n}))
=
\sum_{j=0}^g\sum_{l=0}^{\lfloor\frac{i-j}{2}\rfloor}
\left[\binom{2g}{j}-\binom{2g}{j-2}\right]
\binom{n+i-j-2l-2}{n-2}
\binom{2g+l-1}{2g-1}
\\
+
\sum_{j=g+1}^{2g+1}\sum_{l=0}^{\lfloor\frac{i-j}{2}\rfloor}
\left[\binom{2g}{j-1}-\binom{2g}{j+1}\right]
\binom{n+i-j-2l-2}{n-2}
\binom{2g+1-1}{2g-1}.
\end{multline*} 
Indexing the first sum in the second expression by $r=2g+1-j$ instead of $j$, the difference of binomial coefficents becomes \[\binom{2g}{2g-r}-\binom{2g}{2g-r+2}=\binom{2g}{r}-\binom{2g}{r-2},\] which matches the difference in the first sum. The expression becomes 
\begin{align*}
\sum_{j=0}^g
\bigg[
\binom{2g}{j}&-\binom{2g}{j-2}
\bigg]
\bigg[\sum_{l=0}^{\lfloor \frac{i-j}{2}\rfloor}
\binom{n+i-j-2l-2}{n-2}\binom{2g+l-1}{2g-1}
\\&\qquad+
\sum_{l=0}^{\lfloor\frac{i+j-2g-1}{2}\rfloor}
\binom{n+i+j-2g-2l-3}{n-2}\binom{2g+l-1}{2g-1}
\bigg].
\end{align*} 
Now, if $l>\lfloor\frac{i+j-2g-1}{2}\rfloor$, then $2l>i+j-2g-1$, so that the first binomial vanishes. Thus we may combine the sums to obtain the desired expression.
\end{proof}

\section{Closed orientable surfaces}\label{sec:closedsurfaces}

\subsection{Technical setup} Let $\Sigma=\Sigma_g$ denote the closed surface of genus $g$. Since
\[
H_c^{-*}(\Sigma;\mathbb{Q})\cong \mathbb{Q}\oplus \mathbb{Q}[-1]^{2g}\oplus \mathbb{Q}[-2],
\]
we may write the underlying graded vector space of the Chevalley--Eilenberg complex as 
\[
CE(\mathfrak{g}_{\Sigma})\cong\mathbb{Q}[p,\tilde{a}_1,\ldots,\tilde{a}_g,\tilde{b}_1,\ldots,\tilde{b}_g,v]
\otimes
\Lambda[\tilde{p},{a}_1,\ldots,{a}_g,{b}_1,\ldots,{b}_g,\tilde{v}]
\]
where the variables are in the following degrees and weights:

\vspace{10pt}
\begin{center}
\begin{tabular}{l|cccc}
&degree $0$ & degree $1$ & degree $2$ & degree $3$\\
\hline
weight $1$&$p$&$a_i,b_i$ & $v$ &\\
weight $2$&&$\tilde{p}$&$\tilde{a}_i,\tilde{b}_i$&$\tilde{v}$
\end{tabular}
\end{center}
\vspace{10pt}

The differential $D$ is specified as a coderivation by a map from this symmetric coalgebra to the cogenerators, the nonzero components of which are $D(a_ib_i)=\tilde{p}$ and $D(vx) = \tilde{x}$ for $x$ of weight 1. 

To give a general formula for $D$, we introduce the following operators, which will play a key role in the remainder of the paper.

\begin{defi}
On $CE(\mathfrak{g}_{\Sigma})$ we define the operators
\[
\Delta = \sum_j\partial_{b_j}\partial_{a_j},\qquad
\delta = \sum_{j;\,c\in\{a,b\}} \tilde{c}_j\partial_{c_j},\] where $\partial_u$ is the degree $-1$ operator defined by formal differentiation with respect to $u$.
\end{defi}
Note that $\delta$ is a differential while $\Delta$ is not; moreover, the two commute. Since $\Delta$ is even, we can be relatively cavalier about applying it to elements, but it is necessary to be more careful with the odd operator $\delta$.

In terms of these operators, we have the formula \[D=\tilde p\Delta+\delta\partial_{v}+\frac{\tilde v}{2}\partial^2_v+\tilde p\partial_p\partial_v.\]

Our strategy in making this formula more comprehensible will be to eliminate the last two terms using contracting homotopies. In order to state the end result, we require some terminology. First, as we will see below, the Betti number $\beta_i(B_k(\Sigma_g))$ is zero for $i>k+1$ and independent of $k$ for $i<k$ (indeed, this latter fact is a general phenomenon; see \cite{Church:HSCSM} and \cite{RandalWilliams:HSUCS}). We write $\beta_i^\mathrm{st}(B(\Sigma_g))=\beta_i(B_k(\Sigma_g))$ for any $i<k$ and refer to this number as the $i$th \emph{stable Betti number}. 

In order to keep track of the ranks of graded spaces, we will use Poincar\'{e} series, our conventions regarding which may be found in Section \ref{subsec: conventions}. We write \begin{align*}\poin_{\mathrm{st}}(t)&=\sum_{i=0}^\infty\beta_i^\mathrm{st}(B(\Sigma_g))t^i\\
 \poin_{0}(t)&=\sum_{i=0}^\infty \beta_i(B_i(\Sigma_g))t^i\\
 \poin_{1}(t)&=\sum_{i=1}^\infty \beta_i(B_{i-1}(\Sigma_g))t^i
 \end{align*} 
for the Poincar\'{e} series recording the Betti numbers of interest. We further define the following two subspaces of $CE(\mathfrak{g}_{\Sigma})$:
\begin{align*}
\XX_g&\coloneqq 
\mathbb{Q}[\tilde a_i, \tilde b_i]
\otimes
\Lambda[{a}_i,{b}_i]
\\
\mathcal{K}_g&\coloneqq \left(\ker\delta|_\XX \cap \ker\Delta|_{\XX}\right),
\end{align*} where $i$ runs from 1 to $g$. Since $g$ is often fixed, we will abbreviate the names of these spaces to $\XX$ and $\mathcal{K}$ when confusion is unlikely, but we will have need of the subscript when we allow $g$ to vary. We also use the notation $X$ and $K$ for $\poin_\XX(t)$ and $\poin_{\mathcal{K}}(t)$, respectively. We note the easily verified equalities \[
X_g=\frac{1}{(1-t)^{2g}}=\sum_{i=0}^\infty \binom{2g+i-1}{i}t^i.
\]

Our first main technical result, whose proof occupies Section \ref{section:theoremproof}, provides formulas for the Poincar\'{e} series of interest in terms of these auxiliary subspaces. 

\begin{thm}\label{thm:summary} Poincar\'e series for the Betti numbers of the configuration spaces of $\Sigma$ are given by 
\begin{align*}
\poin_{\mathrm{st}}(t)&=
\frac{1+t^3}{t^2}
\big[(1+t)K  +(-t+t^2)X - 1\big]
\\
\poin_0(t)&=\frac{1}{t}\big[
(1+t^2+t^3)K - 1+t-t^2 + (-t^3+t^4)X
\big]
\\
\poin_{1}(t)&=t^2K.
\end{align*} 
\end{thm}

These are the formulas we use to make our computations. As an illustration of their efficiency, we present the case $g=0$.

\begin{corollary}\label{cor:sphere}
For any $k\geq0$, $$\beta_i(B_k(S^2))=\begin{cases}
1 &\quad i=0 \text{ or } i=3\leq k \text{ or } i=2=k+1\\
0 &\quad \text{else}
\end{cases}$$
\end{corollary}
\begin{proof}
Since $g=0$, $\XX=\mathbb{Q}$, and $\delta=\Delta=0$. Thus \begin{align*}
\poin_{\mathrm{st}}(t)&=1+t^3\\
\poin_0(t)&=1+t^3\\
\poin_{1}(t)&=t^2.
\end{align*} Thus, when $i=0$, there is a one-dimensional contribution from the first series for all $k>0$ and from the second series when $k=0$; when $i=2$, there is a one-dimensional contribution from the third series when $k=1$; when $i=3$, there is a one-dimensional contribution from the first series for all $k>3$ and from the second series when $k=3$; and there are no contributions for $i\notin\{0,2,3\}$. 
\end{proof}

Our second main technical result is a computation of the graded dimension of the auxiliary space $\mathcal{K}_g$. 

\begin{thm}
\phantomsection
\label{thm: Kg calculation}
\begin{enumerate} \item\label{item:Kg dimensions} For $i\ge 3$ and $g\ge 0$, the graded dimensions of $\mathcal{K}_g$ are given by 
\[
\dim_i(\mathcal{K}_g)=\sum_{j=0}^{g-1}\sum_{m=0}^j
(-1)^{g+j+1}\frac{2j-2m+2}{2j-m+2}
\binom{\frac{6j+2i+2g-2m-1-3(-1)^{i+j+g+m}}{4}}{m,2j-m+1}.
\] The special cases for $i\le 2$ are
\begin{align*}
\dim_i(\mathcal{K}_g)=&
\left\{
\begin{array}{ll}
1&i=0\\
0&i=1\\
2g&i=2.
\end{array}\right.
\end{align*}
\item\label{item:high g Kg dimensions} In the range $3\le i\le g+2$, there is the simplification
\[\dim_i(\mathcal{K}_g)=\binom{2g+i-3}{i-1}.\]
\end{enumerate}
\end{thm}

Proving this result will be the object of Section \ref{section:lemmasproof}; for the moment, we concentrate on exploiting it.

\subsection{Results} The three corollaries that follow provide explicit formulas for all Betti numbers of configuration spaces of closed surfaces. The proofs are immediate from Theorems \ref{thm:summary} and \ref{thm: Kg calculation}, together with the formula for $X$ given above. When the degree $i$ is at least five, the situation is ``generic'' and each formula is merely the sum of the formula for $K_g$ given in Theorem \ref{thm: Kg calculation}(\ref{item:Kg dimensions}), suitably reindexed, with a simplification of any summands involving $X$. When $i$ is small, the calculation is easily carried out by hand, in some cases with the aid of Theorem \ref{thm: Kg calculation}(\ref{item:high g Kg dimensions}).

\begin{corollary}\label{cor:unstable k+1}
For $i\ge 5$ and $g\ge 0$ the unstable Betti number $\beta_i(B_{i-1}(\Sigma_g))$ is 
\[
\beta_i(B_{i-1}(\Sigma_g))
=
\sum_{j=0}^{g-1}\sum_{m=0}^j
(-1)^{g+j+1}\frac{2j-2m+2}{2j-m+2}
\binom{\frac{6j+2i+2g-2m-5-3(-1)^{i+j+g+m}}{4}}{m,2j-m+1}.
\]
The special cases for $i<5$ are:
\begin{align*}
\beta_i(B_{i-1}(\Sigma_g))=
\left\{
\begin{array}{ll}
0&i=1\\
1&i=2\\
0&i=3\\
2g&i=4.
\end{array}\right.
\end{align*}
\end{corollary}

\begin{corollary}\label{cor: second unstable betti number}
For $i\ge 5$ and $g\ge 0$, the unstable Betti number $\beta_i(B_{i}(\Sigma_g))$ is
\begin{multline*} 
-\binom{2g+i-4}{2g-2}+\sum_{j=0}^{g-1}\sum_{m=0}^j
(-1)^{g+j+1}\frac{2j-2m+2}{2j-m+2}
\Bigg[
\binom{\frac{6j+2i+2g-2m+1+3(-1)^{i+j+g+m}}{4}}{m,2j-m+1}
\\+\binom{\frac{6j+2i+2g-2m-3+3(-1)^{i+j+g+m}}{4}}{m,2j-m+1}
+\binom{\frac{6j+2i+2g-2m-5-3(-1)^{i+j+g+m}}{4}}{m,2j-m+1}
\Bigg].
\end{multline*}
The special cases for $i<5$ are:
\begin{align*}
\beta_i(B_{i}(\Sigma_g))=
\left\{
\begin{array}{ll}
1&i=0\\
2g&i=1\\
2g^2-g&i=2\\
4&i=3,g=1\\
(4g^3-g+3)/3&i=3,g\ne 1\\
0&i=4,g=0\\
4&i=4,g=1\\
24&i=4,g=2\\
(4g^4+4g^3-g^2+11g)/6&i=4,g>2.
\end{array}\right.
\end{align*}
\end{corollary}

\begin{corollary}\label{cor: stable numbers}
For $i\ge 5$ and $g\ge 0$, the stable Betti number $\beta_i^{\mathrm{st}}(B(\Sigma_g))$ is 
\begin{multline*}
-\binom{2g+i-1}{2g-2}-\binom{2g+i-4}{2g-2}+\sum_{j=0}^{g-1}\sum_{m=0}^j
(-1)^{g+j+1}\frac{2j-2m+2}{2j-m+2}
\Bigg[
\\\binom{\frac{6j+2i+2g-2m+3-3(-1)^{i+j+g+m}}{4}}{m,2j-m+1}
+\binom{\frac{6j+2i+2g-2m+1+3(-1)^{i+j+g+m}}{4}}{m,2j-m+1}
\\+\binom{\frac{6j+2i+2g-2m-3+3(-1)^{i+j+g+m}}{4}}{m,2j-m+1}
+\binom{\frac{6j+2i+2g-2m-5-3(-1)^{i+j+g+m}}{4}}{m,2j-m+1}
\Bigg].
\end{multline*}

The special cases for $i<5$ are:
\begin{align*}
\beta_i^{\mathrm{st}}(B(\Sigma_g))=
\left\{
\begin{array}{ll}
1&i=0\\
2g&i=1\\
0&i=2,g=0\\
3&i=2,g=1\\
2g^2-g&i=2,g>1\\
1&i=3,g=0\\
5&i=3,g=1\\
16&i=3,g=2\\
(4g^3-g+3)/3&i=3,g>2\\
0&i=4,g=0\\
7&i=4,g=1\\
28&i=4,g=2\\
90&i=4,g=3\\
(4g^4+4g^3-g^2+11g)/6&i=4,g>3.
\end{array}\right.
\end{align*}
\end{corollary}

Because the summations over genus in the formulas of Corollaries~\ref{cor:unstable k+1},~\ref{cor: second unstable betti number}, and~\ref{cor: stable numbers} are rather messy, we complement them with two further results demonstrating asymptotic behavior for high genus and for high degree, along with a handful of explicit computations in low genus and/or degree. We begin with a kind of genus stability result, which is immediate from Theorems \ref{thm:summary} and \ref{thm: Kg calculation}(\ref{item:high g Kg dimensions}):

\begin{corollary}\label{cor:genus stability}
In the range $5\le i\le g$, we have:
\[
\beta_i(B_k(\Sigma_g))=\left\{
\begin{array}{rl}
{}\displaystyle\binom{2g+i-2}{i}+\binom{2g+i-5}{i-3}&i\le k
\\\\
\displaystyle\binom{2g+i-5}{i-3}&i=k+1
\\\\
\displaystyle 0&i>k+1.
\end{array}
\right.
\]
\end{corollary}

Next, there is the following consequence of the formulas of Corollaries~\ref{cor:unstable k+1},~\ref{cor: second unstable betti number}, and~\ref{cor: stable numbers}.

\begin{corollary}\label{cor: fixed g polynomials}
For fixed $g$, there are polynomials with rational coefficients of degree $2g-1$ in one variable, $p^{\mathrm{st}}_g$, $q^{\mathrm{st}}_g$, $p_g^{0}$, $q_g^0, p_g^1$, and $q_g^1$, such that the Betti numbers in genus $g$ are given by:
\begin{align*}
\beta_i^{\mathrm{st}}(B(\Sigma_g))&=\left\{\begin{array}{ll}
p^{\mathrm{st}}_g(i) & i\ge 5\text{, odd}
\\
q^{\mathrm{st}}_g(i) & i\ge 6\text{, even.}
\end{array}\right.\\
\beta_i(B_i(\Sigma_g))&=\left\{\begin{array}{ll}
p^0_g(i) & i\ge 5\text{, odd}
\\
q^0_g(i) & i\ge 6\text{, even.}
\end{array}\right.
\\
\beta_i(B_{i-1}(\Sigma_g))&=\left\{\begin{array}{ll}
p^1_g(i) & i\ge 5\text{, odd}
\\
q^1_g(i) & i\ge 6\text{, even.}
\end{array}\right.
\end{align*}
\end{corollary}
\begin{proof}
Inside the summations calculating the Betti number, the only factor with $i$ dependence is of the form $\binom{\frac{i}{2}+a+(-1)^ib}{k,2j-k+1}$ which is a degree $2j+1\le 2g-1$ polynomial in $i$ for fixed parity. Outside the summation, the stable and lower unstable Betti numbers also have a contribution which is polynomial of degree $2g-2$ in $i$.
\end{proof}
Unfortunately, the presentation in the corollaries is not particularly readable, and one might hope for a simplification giving the coefficients of these polynomials in a more comprehensible manner, or at least for a discernable pattern of some sort. One example of such a possible pattern is the observation that for $2\le g\le 20$, the coefficients of $p^{\mathrm{st}}_g$, $q^{\mathrm{st}}_g$, $p^0_g$, and $q^0_g$ are all nonnegative. These polynomials are given for $g\le 3$ in Figure~\ref{fig: table of polynomials in low genus}. To illustrate the growth of the integers involved, we note that $q^{\mathrm{st}}_{5}(i)$, the polynomial for even degree stable Betti numbers in genus $5$, is:
\[\frac{i^{9}}{368640}+\frac{7i^{8}}{81920}+\frac{227i^{7}}{161280}+\frac{1393i^{6}}{92160}+\frac{319i^{5}}{2880}+\frac{24947i^{4}}{46080}+\frac{9331i^{3}}{5760}+\frac{7811i^{2}}{2880}+\frac{407i}{140}+2.
\] 

Figure~\ref{fig: chart of stable values} contains a chart of values of stable Betti numbers for $g\le 6$ and $i\le 43$. We emphasize that with the formulas of this paper, the calculation of such polynomials or Betti numbers for much larger $g$ and $i$ is a computationally trivial task; these charts are included purely to provide convenient examples.

\begin{figure}
\[\arraycolsep=1.4pt\def\arraystretch{2.2}
\begin{array}{rcccc}
&g=0&g=1&g=2&g=3\\
\hline
p^{\mathrm{st}}_g&
0 & 2i-1 & \frac{2i^{3}+3i^{2}+10i+9}{8} & \frac{2i^{5}+15i^{4}+80i^{3}+198i^{2}+302i+363}{192}
\\
q^{\mathrm{st}}_g &
0 & 2i-1 & \frac{2i^{3}+3i^{2}+10i}{8} & \frac{2i^{5}+15i^{4}+80i^{3}+252i^{2}+464i+192}{192}
\\
p^0_g&
0 & \frac{3i-1}{2} & \frac{3i^{3}+i^{2}+13i+31}{16} & \frac{3i^{5}+22i^{4}+162i^{3}+500i^{2}+603i+630}{384}
\\
q^0_g&
0 & \frac{3i-4}{2} & \frac{3i^{3}+4i^{2}+28i}{16} & \frac{3i^{5}+19i^{4}+120i^{3}+500i^{2}+1200i+384}{384}
\\
p^1_g&
0 & \frac{i-3}{2} & \frac{i^{3}+i^{2}-9i-9}{16} & \frac{i^{5}+4i^{4}-2i^{3}+32i^{2}+i-420}{384}
\\
q^1_g&
0 & \frac{i}{2} & \frac{i^{3}-2i^{2}+16}{16} & \frac{i^{5}+7i^{4}-8i^{3}-76i^{2}+112i+384}{384}
\end{array}
\]
\caption{The first few polynomials from Corollary~\ref{cor: fixed g polynomials}. These give both stable and unstable Betti numbers for $i\ge 5$.}\label{fig: table of polynomials in low genus}
\end{figure}
\begin{figure}
\[
\begin{array}{lccccccc}
i\backslash g & 0 & 1 & 2 & 3 & 4 & 5 & 6\\\hline\hline
0 & 1 & 1 & 1 & 1 & 1 & 1 & 1\\
1 & 0 & 2 & 4 & 6 & 8 & 10 & 12\\
2 & 0 & 3 & 6 & 15 & 28 & 45 & 66\\
3 & 1 & 5 & 16 & 36 & 85 & 166 & 287\\
4 & 0 & 7 & 28 & 90 & 218 & 505 & 1013\\
5 & 0 & 9 & 48 & 169 & 532 & 1332 & 3069\\
6 & 0 & 11 & 75 & 335 & 1098 & 3300 & 8294\\
7 & 0 & 13 & 114 & 569 & 2289 & 7227 & 20878\\
8 & 0 & 15 & 162 & 979 & 4187 & 15587 & 47762\\
9 & 0 & 17 & 225 & 1531 & 7748 & 30294 & 105963\\
10 & 0 & 19 & 300 & 2396 & 13034 & 58860 & 216281\\
11 & 0 & 21 & 393 & 3520 & 22079 & 105118 & 436150\\
12 & 0 & 23 & 501 & 5151 & 34866 & 188319 & 818752\\
13 & 0 & 25 & 630 & 7211 & 55223 & 315369 & 1530869\\
14 & 0 & 27 & 777 & 10039 & 82965 & 529718 & 2693703\\
15 & 0 & 29 & 948 & 13529 & 124690 & 842884 & 4736380\\
16 & 0 & 31 & 1140 & 18125 & 179921 & 1343826 & 7912036\\
17 & 0 & 33 & 1359 & 23689 & 259302 & 2050653 & 13221792\\
18 & 0 & 35 & 1602 & 30784 & 361900 & 3132029 & 21159269\\
19 & 0 & 37 & 1875 & 39236 & 504021 & 4615128 & 33879846\\
20 & 0 & 39 & 2175 & 49741 & 684067 & 6800508 & 52294099\\
21 & 0 & 41 & 2508 & 62085 & 926002 & 9727432 & 80742936\\
22 & 0 & 43 & 2871 & 77111 & 1227304 & 13904838 & 120830579\\
23 & 0 & 45 & 3270 & 94561 & 1622011 & 19387707 & 180821641\\
24 & 0 & 47 & 3702 & 115439 & 2106363 & 27001767 & 263434743\\
25 & 0 & 49 & 4173 & 139439 & 2727348 & 36822006 & 383668154\\
26 & 0 & 51 & 4680 & 167740 & 3479594 & 50140352 & 545978070\\
27 & 0 & 53 & 5229 & 199984 & 4426415 & 67056804 & 776480287\\
28 & 0 & 55 & 5817 & 237539 & 5560388 & 89530551 & 1082270541\\
29 & 0 & 57 & 6450 & 279991 & 6965069 & 117692377 & 1507214918\\
30 & 0 & 59 & 7125 & 328911 & 8630475 & 154433796 & 2062327850\\
31 & 0 & 61 & 7848 & 383825 & 10664900 & 199922976 & 2818996389\\
32 & 0 & 63 & 8616 & 446521 & 13055217 & 258327002 & 3793935067\\
33 & 0 & 65 & 9435 & 516461 & 15939574 & 329858905 & 5100110873\\
34 & 0 & 67 & 10302 & 595664 & 19301036 & 420398901 & 6762379052\\
35 & 0 & 69 & 11223 & 683524 & 23313381 & 530213298 & 8955099361\\
36 & 0 & 71 & 12195 & 782305 & 27955117 & 667445528 & 11714562366\\
37 & 0 & 73 & 13224 & 891329 & 33442128 & 832424580 & 15303928783\\
38 & 0 & 75 & 14307 & 1013119 & 39747526 & 1036240128 & 19775312348\\
39 & 0 & 77 & 15450 & 1146921 & 47136517 & 1279294291 & 25517963824\\
40 & 0 & 79 & 16650 & 1295531 & 55575883 & 1576461699 & 32605592304\\
41 & 0 & 81 & 17913 & 1458115 & 65388148 & 1928229150 & 41603514029\\
42 & 0 & 83 & 19236 & 1637756 & 76532730 & 2354275836 & 52614541111\\
43 & 0 & 85 & 20625 & 1833536 & 89398287 & 2855185938 & 66446126460\\
\end{array}
\]
\caption{Stable Betti numbers for low degree and genus}\label{fig: chart of stable values}
\end{figure}

\section{Proof of Theorem \ref{thm:summary}}\label{section:theoremproof}

Recall from Section~\ref{sec:closedsurfaces} that the differential on $CE(\mathfrak{g}_\Sigma)$ is given by the formula \[D=\tilde p\Delta+\delta\partial_{v}+\frac{\tilde v}{2}\partial^2_v+\tilde p\partial_p\partial_v.\] Our first simplification is to eliminate the term $\frac{\tilde v}{2}\partial^2_v$ using a contracting homotopy. First, we recall that the spaces of interest are

\begin{align*}
\XX_g&\coloneqq 
\mathbb{Q}[\tilde a_i, \tilde b_i]
\otimes
\Lambda[{a}_i,{b}_i]
\\
\mathcal{K}_g&\coloneqq \left(\ker\delta|_\XX \cap \ker\Delta|_{\XX}\right),
\end{align*} with Poincar\'{e} series $X$ and $K$, respectively. We further define two auxiliary subspaces of $\XX_g$:

\begin{align*}\YY_g&\coloneqq 
\mathbb{Q}[p,\tilde a_i, \tilde b_i]
\otimes
\Lambda[{a}_i,{b}_i]\\
\ZZ_g&\coloneqq 
\mathbb{Q}[p,\tilde a_i, \tilde b_i]
\otimes
\Lambda[\tilde p,{a}_i,{b}_i].\end{align*} As before, $i$ runs from 1 to $g$, and we omit the subscript when the genus is unambiguous.

\begin{lemma}\label{lemma:chainhomotopy}
Define
\begin{enumerate}
\item a retraction $CE(\mathfrak{g}_\Sigma)\to \ZZ\oplus v\ZZ$, which, for $q\in \ZZ$, takes $\tilde{v}q$ to 
\[-2v\delta(q)-2v\tilde{p}\partial_p q,\] and
\item a degree $1$ chain homotopy from $CE(\mathfrak{g}_M)$ to itself, which, for $q$ in $\ZZ$, takes $v^n\tilde vq$ to
\[
\frac{2n!}{(n+2)!}v^{n+2}q.\]
\end{enumerate}
Then the inclusion of $\ZZ\oplus v\ZZ$ into $CE(\mathfrak{g})$ along with the retraction and chain homotopy constructed above constitute the data of a deformation retraction.
\end{lemma}
The proof is a direct computation.

\begin{remark}
It follows from this step alone that $H_i(B_k(\Sigma_g);\mathbb{Q})=0$ for $i>k+1$.
\end{remark}

Next, we eliminate the term $\tilde p\partial_p\partial_v$ from the differential by deforming $\ZZ\oplus v\ZZ$ onto the subspace $\YY\oplus v\tilde{p}\YY\oplus v\XX$ endowed with a twisted differential. See Figure~\ref{fig: inclusion diagram} for a graphical summary of the differentials and one of the maps used in the deformation retraction.

\begin{lemma}\label{lemma:chainhomotopy2}
Define
\begin{enumerate}
\item a degree $-1$ linear operator $d$ on $\YY\oplus v\tilde{p}\YY\oplus v\XX$ by 
\begin{align*}
d|_\YY(q)&=
-\frac{p\Delta}{n+1}(v\tilde{p}\Delta + \delta)q
\\
d|_{v\tilde{p}\YY}(q)&=
-\frac{p\Delta}{n+1}\delta(q) 
\\
d|_{v\XX}(q)&=\left(\frac{1}{v}\delta+\tilde p\Delta\right)q,
\end{align*} where $n=n(q)$ is the largest nonnegative integer such that $q=p^nq'$;
\item a linear map $f:\YY\oplus v\tilde p\YY\oplus v\XX\to \ZZ\oplus v\ZZ$ by \begin{align*}
f|_\YY(q)&=\left(1-\frac{p}{n+1}v\Delta\right)q\\
f|_{v\tilde p\YY}(q)&=\left(1-\frac{p}{\tilde p(n+1)}\delta\right)q\\
f|_{v\XX}(q)&=q;
\end{align*}
\item a linear map $g:\ZZ\oplus v\ZZ\to \YY\oplus v\tilde p\YY\oplus v\XX$ by \begin{align*}
g|_{\YY\oplus v\tilde p \YY}(q)&=q\\
g|_{v\YY}(q)&=\begin{cases}
q&\quad n=0\\
0&\quad \text{else}
\end{cases}\\
g|_{\tilde p \YY}(q)&=-\frac{p}{n+1}\left(v\Delta-\frac{1}{\tilde p}\delta\right)q
\end{align*}
\item a degree $1$ linear operator on $\ZZ\oplus v\ZZ$ by stipulating that \begin{align*}
H|_{\tilde p\YY}&=\frac{p}{\tilde p(n+1)}v
\end{align*} and extending by zero.
\end{enumerate}
Then $d$ is a differential, $f$ and $g$ are chain maps, and $H$ is a chain homotopy $\id\sim fg.$
\end{lemma}

\begin{proof} The proof is a direct, albeit tedious, computation. We leave to the reader all but the verification that $g$ is a chain map. For this, we have \begin{align*}
dg|_{\YY}&=-\frac{p\Delta}{n+1}(v\tilde{p}\Delta + \delta)\\
dg|_{v\tilde p\YY}&=-\frac{p\Delta}{n+1}\delta\\
dg|_{v\YY}&=\begin{cases}
\frac{1}{v}\delta+\tilde p\Delta&\quad n=0\\
0&\quad \text{else}
\end{cases}\\
dg|_{\tilde p \YY}&=\frac{p\Delta}{n+2}\delta\frac{p}{n+1}v\Delta -\frac{p\Delta}{n+2}(v\tilde{p}\Delta+\delta)\frac{p}{n+1}\frac{1}{\tilde p}\delta  \\
&=\frac{p^2\Delta}{(n+2)(n+1)}(\delta v\Delta-v\tilde{p}\Delta\frac{1}{\tilde p}\delta)
\\
&=0\\
gD|_{\YY}&=-\frac{p}{n+1}\left(v\Delta-\frac{1}{\tilde p}\delta\right)\tilde p\Delta\\
&=-\frac{p\Delta}{n+1}(v\tilde p\Delta+\delta)\\
gD|_{v\tilde p\YY}&=-\frac{p}{n+1}\left(v\Delta-\frac{1}{\tilde p}\delta\right)\frac{\delta}{v}\\
&=-\frac{p\Delta}{n+1}\delta\\
gD|_{v\YY}(q)&=\left(\frac{1}{v}\delta+\tilde p\Delta\right)q-\frac{p}{n+1}\left(v\Delta-\frac{1}{\tilde p}\delta\right)\frac{\tilde p}{v}\partial_pq\\
&=\left(\frac{1}{v}\delta+\tilde p\Delta\right)\left(1-\frac{p}{n+1}\partial_p\right)q\\
&=\begin{cases}
(\frac{1}{v}\delta+\tilde p\Delta)q&\quad n=0\\
0&\quad\text{else}
\end{cases}\\
gD|_{\tilde p\YY}&=0,
\end{align*} 
where we have made use of the fact that if $n=0$, then $\partial_pq=0$, while if $n\ne 0$, we have
\[\left(\frac{p}{n+1}\partial_p\right)q=\frac{p}{n(\partial_pq)+1}\partial_pq=\frac{p}{n(q)}\partial_pq=
q.\]
Thus, $gD=dg$.

\end{proof}

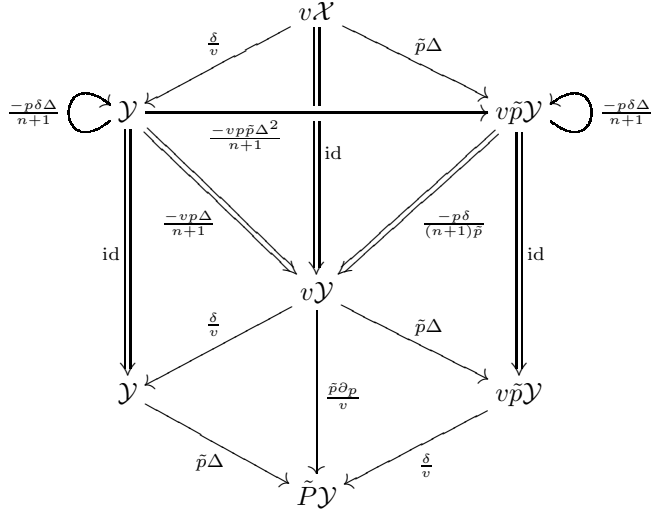
\begin{figure}
\[
\xymatrix{
&&v\XX\ar[dll]_{\frac{\delta}{v}}\ar[drr]^{\tilde{p}\Delta}\ar@{=>}[ddd]|(.36)\hole^{\id}
\\
\YY\save!L(.5)\ar@(dl,ul)^{\frac{-p\delta\Delta}{n+1}}\restore\ar[rrrr]_(.30){\frac{-vp\tilde{p}\Delta^2}{n+1}}\ar@{=>}[ddrr]_{\frac{-vp\Delta}{n+1}}\ar@{=>}[ddd]_{\id}&&&&v\tilde{p}\YY\ar@{=>}[ddll]^{\frac{-p\delta}{(n+1)\tilde{p}}}\ar@{=>}[ddd]^{\id}\save!R(.7)\ar@(dr,ur)_{\frac{-p\delta\Delta}{n+1}}\restore
\\\\
&&v\YY\ar[dll]_{\frac{\delta}{v}}\ar[drr]^{\tilde{p}\Delta}\ar[dd]^{\frac{\tilde{p}\partial_p}{v}}
\\	
\YY\ar[drr]_{\tilde{p}\Delta}&&&&v\tilde{p}\YY\ar[dll]^{\frac{\delta}{v}}
\\
&&\tilde{P}\YY
}
\]
\caption{The inclusion $f$ of $\YY\oplus v\tilde{p}\YY\oplus v\XX$ with the twisted differential into $\ZZ\oplus v\ZZ$. Single arrows indicate differentials in the domain and codomain of $f$ and double arrows indicate components of $f$.}
\label{fig: inclusion diagram}
\end{figure}

\begin{corollary}\label{corollary:stable complex}
For $i<k$,
\[
H_i(B_k(\Sigma_g);\mathbb{Q})\cong H_i(\XX\oplus \XX[3],\,d_\mathrm{st})
\]
where $d_\mathrm{st}$ has components $\Delta^2[3]: \XX\to \XX[3]$, $\delta\Delta: \XX\to \XX$, and $-\delta\Delta:\XX[3]\to \XX[3]$. In particular, the dimension of this vector space is independent of $k$. 
\end{corollary}
\begin{proof}
The degree less than weight subspace of $\YY\oplus v\tilde p\YY\oplus v\XX$ is $$p\YY\oplus pv\tilde p\YY\cong \bigoplus_{n=1}^\infty p^n\XX\oplus p^nv\tilde p \XX.$$ It is clear from the formulas that the respective dimensions of the kernel and image of $$d:p^n\XX\oplus p^nv\tilde p\XX\to p^{n+1}\XX\oplus p^{n+1}v\tilde p \XX$$ are independent of $n$ and therefore of $k$. Thus, for example, the Betti numbers of $B_k(\Sigma_g)$ for $i<k$ and any $k$ are computed by the complex $$
p\XX\oplus p\XX[3]\overset{-\frac{p}{2}d_\mathrm{st}}{\longrightarrow} p^2\XX\oplus p^2\XX[3]\overset{-\frac{p}{3}d_\mathrm{st}}{\longrightarrow}\cdots,
$$ which is isomorphic to $(\XX\oplus \XX[3],d_\mathrm{st})$ after the change of variables $p^n\mapsto (-1)^n\frac{p^n}{n!}$.
\end{proof}

Corollary~\ref{corollary:stable complex} asserts that we may calculate stable Betti numbers using the complex

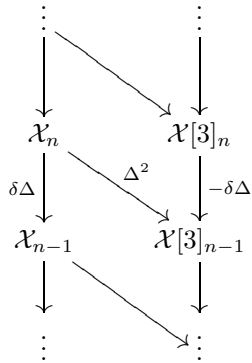
\[
\xymatrix{
\vdots\ar[d]\ar[dr]&\vdots\ar[d]
\\
\XX_n\ar[d]_{\delta\Delta}\ar[dr]^{\Delta^2}&
\XX[3]_{n}\ar[d]^{-\delta\Delta}
\\
\XX_{n-1}\ar[d]\ar[dr]&\XX[3]_{n-1}\ar[d]
\\
\vdots&\vdots
}
\]

We will use two more lemmas for the proof of Theorem~\ref{thm:summary}. We write $\XX_{>0}\subset \XX$ for the subspace of polynomials with no constant term. 

\begin{lemma}\label{lemma: acyclic}
The cochain complex $(\XX_{>0},\delta)$ is acyclic. More specifically, there is an operator $H$ on $\XX$ that restricts to a cochain nullhomotopy of $(\XX_{<0},\delta)$, also called $H$, and this operator has the property that $\Delta H\Delta H$ and $H\Delta H\Delta$ differ by an invertible linear transformation.
\end{lemma}
\begin{proof}
It is easily seen that the operator \[h=\sum_{c}c\partial_{\tilde c}\] on $\XX$ has the property that $(\delta h+h\delta)\alpha$ is a nonzero scalar multiple of $\alpha$ for every nonzero monomial $\alpha\in\XX_{>0}$, where $c$ runs over the set $\{a_1,\ldots, a_g,b_1,\ldots, b_g\}$.  We obtain $H$ by setting $H(\alpha)=\frac{h(\alpha)}{\deg \alpha}$ where $\deg \alpha$ is the polynomial degree of the monomial $\alpha$ (not its homological degree), and $H|_{\XX_0}=h|_{\XX_0}=0$.

To see the desired behavior with respect to $\Delta$, first calculate that
\[
[h,\Delta]=\sum_j \partial_{a_j}\partial_{\tilde{b}_j}-\partial_{b_j}\partial_{\tilde{a}_j}.
\]
Then both $\Delta$ and $[h,\Delta]$ are partial differentiation operators so $[[h,\Delta],\Delta]=0$, or in other words $\Delta h\Delta = \frac{1}{2}(\Delta^2 h + h\Delta^2)$. Then since $h^2=0$, we have
\[
\Delta h\Delta h = \frac{1}{2}(h\Delta^2 h + \Delta^2 h^2)=\frac{1}{2}(h\Delta^2 h + h^2\Delta^2 )
=
h\Delta h\Delta.
\]
In polynomial degree at most $4$, $h\Delta h\Delta$ is manifestly zero since $h$ kills anything of nonpositive polynomial degree, and so the same is true for $\Delta h\Delta h$. Then $H\Delta H\Delta = 0 = \Delta H\Delta H$ on elements of $\XX$ of degree at most $4$.

In polynomial degree greater than $4$, since $\Delta$ lowers degree by two, applying $\Delta H\Delta H$ to a monomial of degree $p$ yields $\frac{p-4}{p}H\Delta H\Delta$.
\end{proof}

Acyclicity will be used immediately; the specific commutation relation between $H$ and $\Delta$ will be used below in the proof of Lemma~\ref{lemma: good recurrence}.

\begin{lemma}
There is an isomorphism $\ker d_\mathrm{st}\cong \ker(\delta\Delta\oplus\delta\Delta[3])$ of bigraded vector spaces.
\end{lemma}
\begin{proof}
We have a canonical isomorphism $\XX\cong\XX_{>0}\oplus\mathbb{Q}$, which we write as $q\mapsto (f(q),c(q))$. Clearly $\im\delta\subset\XX_{>0}$, and we fix a section $s:\im \delta\to \XX_{>0}$ of $\delta$.

Now, observe that if $(q,r)\in \XX\oplus \XX[3]$ is annihilated by $d_\mathrm{st}$, then $\Delta q\in\ker\delta$, and \begin{align*}\delta\Delta(sf(\Delta q)-r)
&=\Delta\delta sf(\Delta q)-\delta\Delta r\\
&=\Delta f(\Delta q)-\delta\Delta r\\
&=\Delta^2q-\Delta c(\Delta q)-\delta\Delta r\\
&=\Delta^2 q-\delta\Delta r
\end{align*}
which is the projection to the second factor of $d_\mathrm{st}(q,r)$ and hence zero. 
Thus, the assignment $(q,r)\mapsto (q, sf(\Delta q)-r)$ defines a linear map $\ker d_\mathrm{st}\to \ker (\delta\Delta\oplus\delta\Delta[3])$. This map is an isomorphism with inverse given by the same formula.
\end{proof}

\begin{lemma}\label{lemma: ker dD in terms of K}
The Poincar\'e series for $\ker\delta\Delta$ on $\XX$ satisfies
\[
\poin_{\ker\delta\Delta}(t)=
\frac{1}{t(1+t)}\big[
(1+t)K - 1 + t^2X
\big]
\]
\end{lemma}
\begin{proof}
Consider $x\in \ker\delta\Delta\subset \XX$. Then $\delta x$ is in $\ker\Delta\cap \im\delta$ which differs from $\ker\Delta\cap \ker\delta$ only in that the latter also contains the class of $1$. Furthermore, if $\delta x =\delta x'$ then $x-x'$ is in $\ker\delta$. Then
\[
\poin_{\ker\delta\Delta}(t)=
\frac{1}{t}(K-1) + \poin_{\ker \delta}(t).
\]
Lemma~\ref{lemma: acyclic} implies that $\poin_{H(\XX,\delta)}(t)=1$, which implies by the formula for the Poincar\'e series of the homology of a complex that
\[
\poin_{\ker\delta}(t)=\frac{1+tX}{1+t},
\]
and then expansion completes the proof.
\end{proof}

\begin{proof}[Proof of Theorem \ref{thm:summary}]
We have \begin{align*}\poin_\mathrm{st}(t)&=\poin_{H(\XX\oplus\XX[3],d_\mathrm{st})}(t)\\&=\poin_{H(\XX\oplus\XX[3],\delta\Delta\oplus\delta\Delta[3])}(t)\\&=\poin_{H(\XX,\delta\Delta)}(t)+\poin_{H(\XX[3],\delta\Delta)}(t)\\&=(1+t^3)\poin_{H(\XX,\delta\Delta)}(t).\end{align*} 
Then combining this formula with
\[
\poin_{H(\XX,\delta\Delta)}(t)=\frac{1+t}{t}\poin_{\ker\delta\Delta}(t)-\frac{1}{t}X
\]
and Lemma~\ref{lemma: ker dD in terms of K} yields the desired result.

The unstable Betti number $\beta_i(B_i(\Sigma_g))$ is computed by taking kernel modulo image in the diagram $$\xymatrix{
&v\XX\ar[dl]_-{\frac{1}{v}\delta}\ar[dr]^-{\tilde p\Delta}\\
\XX\ar[d]_-{\delta\Delta}\ar[drr]^-{\Delta^2}&&v\tilde p\XX\ar[d]^-{-\delta\Delta}\\
p\XX&&pv\tilde p\XX
}$$ so that \begin{align*}
\poin_{0}(t)&=\poin_{\ker d_\mathrm{st}}(t)-\poin_{\im (\frac{1}{v}\delta\oplus\tilde p\Delta)}(t)\\
&=\poin_{\ker \delta\Delta\oplus \ker\delta\Delta[3]}(t)-t^{-1}(t^2X-t^2K)\\
&=(1+t^3)\poin_{\ker\delta\Delta}(t)-t(X-K).
\end{align*}
In this case, applying Lemma~\ref{lemma: ker dD in terms of K} and conducting algebraic simplifications yields the result.

Since the unstable Betti number $\beta_i(B_{i-1}(\Sigma_g))$ is computed as the kernel of $\frac{1}{v}\delta\oplus\tilde p\Delta$ acting on $v\XX$, the third claim is immediate.
\end{proof}

\section{Proof of Theorem~\ref{thm: Kg calculation}}\label{section:lemmasproof}

The purpose of this section is to calculate the Poincar\'e series of $\mathcal{K}$, the simultaneous kernel of $\delta$ and $\Delta$, which is used in Section~\ref{sec:closedsurfaces} to calculate the stable and unstable Betti numbers in the case of a closed surface. 

For $g\geq0$ and $n>0$ we write \[\mathcal{V}(g,n)=\{(q,r)\in \XX_g\oplus \XX_g[3]:\Delta^n q = \delta\Delta^{n-1}r,\,\Delta^n r = 0\},\] and we adopt the convention that $\mathcal{V}(g,0)=0$. We also set $\mathcal{S}=\mathbb{Q}[\tilde a,\tilde b],$ where $\tilde a$ and $\tilde b$ have degree 2. We write $V_{g,n}$ and $S$ for the respective Poincar\'{e} series.

These spaces are related to the spaces of interest as follows:

\begin{lemma}\label{lemma:exact_sequence}
For $g\geq0$, there is an exact sequence \[0\longrightarrow \mathcal{K}_{g+1}[1]\longrightarrow \mathcal{S}\otimes \mathcal{V}(g,1)[1]\longrightarrow \mathcal{K}_{g},\] where the rightmost arrow is the composite \[\mathcal{S}\otimes \mathcal{V}(g,1)[1]\xrightarrow{\tilde a,\tilde b= 0}\mathcal{V}(g,1)[1]\xrightarrow{(q,r)\mapsto \delta q}\mathcal{K}_g.\]
\end{lemma}

The proof (as well as the definition of the leftmost map) uses a substantial amount of new notation and is deferred to the end of the section. For the present, we derive the following consequence:

\begin{lemma}\label{lemma: good recurrence}
For $g\geq0$, the Poincar\'e series $K_{g+1}$ satisfies the recurrence relation 
\[
tK_{g+1}=1+ t^3 + tSV_{g,1}- K_{g}.
\]
\end{lemma}
\begin{proof}
In light of Lemma \ref{lemma:exact_sequence}, the bulk of the lemma will be established after showing that the map \[\mathcal{V}(g,1)[1]\xrightarrow{(q,r)\mapsto \delta q}\mathcal{K}_g\] is surjective in degrees 4 and higher. For this, we recall from Lemma~\ref{lemma: acyclic} that there is a cochain nullhomotopy $H$ of $(\XX_{>0},\delta)$ such that $\Delta H\Delta H$ and $H\Delta H\Delta$ differ by an invertible linear transformation. Now, given $x\in \mathcal{K}_g$ of degree at least 4, the pair $(Hx,H\Delta Hx)$ lies in $\mathcal{V}(g,1).$ Indeed, both $x$ and $\Delta Hx$ lie in $\XX_{>0}$ by our assumption on the degree of $x$, so we have \begin{align*}\delta H\Delta Hx&=\Delta Hx - H\delta\Delta Hx\\
&=\Delta Hx - H\Delta\delta Hx\\
&=\Delta Hx - H\Delta x+H\Delta H\delta x\\
&=\Delta Hx,
\end{align*} and \[\Delta x=0\implies H\Delta H\Delta x=0\implies \Delta H\Delta Hx=0.\] Since we also have \[\delta Hx=x - H\delta x=x,\] we conclude that $x$ lies in the image of the indicated map, establishing surjectivity. This in turn implies that the rightmost map in the exact sequence of Lemma \ref{lemma:exact_sequence} is surjective, which proves the claimed equality of Poincar\'{e} series above degree 3. 

From Lemma~\ref{lemma: acyclic}, it follows $\mathcal{K}_g$ is spanned in degree 2 by the elements of the form $\tilde c_i$, so that $(c_i,0)\in \mathcal{V}(g,1)$ is a preimage and we have surjectivity in this degree as well. Thus, the claimed equality holds in degree 2. Since the kernel $\mathcal{K}_g$ vanishes in degree $1$ and is one dimensional in degree $0$, it remains to establish the equality in degree 3.

For this, we make use of the modified sequence \[0\longrightarrow\mathcal{K}_{g+1}[1]\longrightarrow \mathcal{S}\otimes\mathcal{V}(g,1)[1]\longrightarrow\mathcal{K}_g/(\tilde a_1 b_1- a_1\tilde b_1).\] If $(q,r)$ is degree $2$ in $\mathcal{V}(g,1)$ and $\delta q=\tilde a_1 b_1- a_1\tilde b_1$, then $q=a_1b_1$ up to the kernel of $\delta$. In degree $2$ this kernel is spanned by $\tilde{a}_i$ and $\tilde{b}_i$, which are also in the kernel of $\Delta$. Then $\delta r = \Delta q = \Delta(a_1b_1)=1$, which is impossible, and we conclude that this modified sequence is still exact. Thus, it suffices to show surjectivity in degree 3 of the map \[\mathcal{V}(g,1)[1]\xrightarrow{(q,r)\mapsto \delta q}\mathcal{K}_g/(\tilde a_1 b_1- a_1\tilde b_1). \] But it is easy to see that, in degree 3, $\mathcal{K}_g=\ker \delta$ is spanned by elements of the form $\tilde c_ic_j-c_i\tilde c_j$, so that the quotient is spanned by classes of the form $\tilde c_ic_j-c_i\tilde c_j$ with $i\neq j$ and $\tilde a_ib_i-a_i\tilde b_i-\tilde a_1b_1+a_1\tilde b_1$. The proof is completed by noting that $(c_ic_j,0)\in \mathcal{V}(g,1)$ is a preimage in the former case and $(a_ib_i-a_1b_1,0)\in \mathcal{V}(g,1)$ is a preimage in the latter case.
\end{proof}

The final missing ingredient is an explicit description of $V_{g,1}$, which we again obtain recursively.

\begin{lemma}\label{lemma:Vrecurrence}
The Poincar\'e series $V_{g,n}$ satisfies the recurrence relation: 
\[
V_{g+1,n} = \xys{}(
V_{g,n+1}+
2tV_{g,n}+
t^2V_{g,n-1})
\]
for $n\ge 1$ and $g\ge 0$.
\end{lemma}

The proof of this recurrence is also deferred to the end of the section.

\begin{corollary}\label{cor:Vsolution}
For $g,n\ge 0$, we have
\[
V_{g,n} =
\xys{}^g(1+t^3)\left(
\sum_{j=0}^{g+n-1} t^j\left(\binom{2g}{j}-\binom{2g}{j-2n}
\right)
\right).
\]
\end{corollary}
\begin{proof} Note that the recursion of Lemma~\ref{lemma:Vrecurrence} may be extended to nonpositive $n$ if we formally set $V_{g,-n}=-t^{-2n}V_{g,n}$. Using the recursion, we write
\[
V_{g,n} = \xys{}^g\sum_{j=0}^{2g} t^j \binom{2g}{j}V_{0,n+g-j}
.\]
Next we note that $V_{0,n}=1+t^3$ for $n>0$ and thus $V_{g,n}$ can be calculated as 
\[
V_{g,n} =\xys{}^g\left(\sum_{j=0}^{g+n-1}(1+t^3) t^j\binom{2g}{j} - \sum_{j=g+n+1}^{2g} (1+t^3)t^{2g+2n-j}\binom{2g}{2g-j}\right).\]

Change of index yields the result.
\end{proof}

This expression admits a convenient simplification when the genus is high relative to the degree.

\begin{corollary}\label{cor: stable value of Vg1}
There is the congruence
\[
V_{g,1}\equiv \frac{(1+t^3)(1-t^2)}{(1-t)^{2g}}\pmod{t^{g+2}}.
\]
\end{corollary}
\begin{proof}
Modulo $t^{g+2}$, we have
\begin{align*}
V_{g,1}&\equiv
S^g(1+t^3)\left(\sum_{j=0}^\infty t^j\left(\binom{2g}{j}-\binom{2g}{j-2}\right)\right)
\\
&\equiv\frac{(1+t^3)(1-t^2)(1+t)^{2g}}{(1-t^2)^{2g}}\pmod{t^{g+2}},
\end{align*}
as desired (note that for the $t^{g+1}$ coefficient, $\binom{2g}{g+1}-\binom{2g}{g-1}$ is identically zero).
\end{proof}

Finally, we are ready to calculate the coefficients of the Poincar\'e series $K_g$ explicitly.
\begin{proof}[Proof of Theorem~\ref{thm: Kg calculation}]

To prove (\ref{item:Kg dimensions}), we use the recurrence relation 
of Lemma~\ref{lemma: good recurrence} repeatedly, picking up factors of $1+t^3$ and $tSV_{i,1}$:
\begin{align*}
K_g &= \frac{1+t^3}{t}+ SV_{g-1,1} -\frac{K_{g-1}}{t}
\\
&=
\frac{1+t^3}{t}+ SV_{g-1,1} - \frac{1+t^3}{t^2} - \frac{SV_{g-2,1}}{t}+\frac{K_{g-2}}{t^2}
\\&=\cdots
\end{align*}
Ending this expansion with $K_1=SV_{0,1}$, we can combine the $1+t^3$ terms and achieve
\[
K_{g}=(1-t+t^2)(1 -(-t)^{1-g})+\sum_{j=0}^{g-1}\frac{SV_{j,1}}{(-t)^{g-1-j}}
\]
valid for $g>0$.

The first term has order at most $t^2$, so for $i\ge 3$,
$\dim_i(\mathcal{K}_g)$ is the degree $i$ term of 
\[
\sum_{j=0}^{g-1}\frac{SV_{j,1}}{(-t)^{g-1-j}}.
\]
Substituting the formula from Corollary~\ref{cor:Vsolution}, this is the degree $i$ term of
\[
\sum_{j=0}^{g-1}\sum_{m=0}^j(-1)^{g+1+j}S^{j+1}(1+t^3)t^{j+m+1-g}\left(\binom{2j}{m}-\binom{2j}{m-2}\right).
\]
For a given choice of degree $i$, the contribution from the $(j,m)$ summand is then $(-1)^{g+1+j}\left(\binom{2j}{m}-\binom{2j}{m-2}\right)$ times the degree $i+g-j-m-1$ coefficient of $(1+t^3)S^{j+1}$. The terms of the Poincar\'e series $S^{j+1}$ are familiar from Section~\ref{sec: non-orientable and open surfaces}. To wit, we have
\[
S^{j+1}=\sum_{i\geq0\,\,\mathrm{ even}}
\binom{2j+\frac{i}{2}+1}{2j+1}t^i
 \] 
so that 
\begin{align*}
(1+t^3)S^{j+1}
&=\sum_{i\geq0\,\,\mathrm{even}}\binom{2j+\frac{i}{2}+1}{2j+1}t^i+\sum_{i\geq0\,\,\mathrm{odd}}\binom{2j+\frac{i-3}{2}+1}{2j+1}t^i
\\
&=\sum_{i\geq0}\binom{2j+\frac{i}{2}+\frac{1}{4} + \frac{3(-1)^{i}}{4}}{2j+1}t^i.
\end{align*} We then have the following equation, valid for $g>0$ and $i\ge 3$:
\begin{align*}
\dim_i(\mathcal{K}_g)=&
\sum_{j=0}^{g-1}\sum_{m=0}^j
(-1)^{g+j+1}\left(\binom{2j}{m}-\binom{2j}{m-2}\right)
\\
&\qquad\qquad\left(\binom{\frac{1}{2}\left(3j+i+g-m-\frac{1}{2}-\frac{3(-1)^{i+j+g+m}}{2}\right)}{2j+1}\right),
\end{align*}
which yields (\ref{item:Kg dimensions}) by a direct computation of the individual summands. The special cases for $i\le 2$ can be seen by the formula further above or by inspection. For the special case $g=0$, the space $\mathcal{K}_g$ is one dimensional, concentrated in degree $0$, which agrees with the formulas of the theorem since the sum for $i\ge 3$ is empty for $g=0$.

For (\ref{item:high g Kg dimensions}), we proceed by induction on $i$, using the recurrence relation of Lemma~\ref{lemma: good recurrence} in the following form:
\begin{align*}
K_g = 1+t^3 + tSV_{g,1} - tK_{g+1}
\end{align*}
For the base case, according to Theorem~\ref{thm: Kg calculation}(\ref{item:Kg dimensions}), the degree $2$ coefficient of $K_{g+1}$ is $2g+2$. Then to determine the degree $3$ coefficient of $K_g$ we need only determine the degree $2$ coefficients of $SV_{g,1}$.  By Corollary~\ref{cor:Vsolution}, modulo $t^3$, $SV_{g,1}$ is
\[
(1+2t^2)^{g+1}\left(
\binom{2g}{0} + \binom{2g}{1}t + \left(\binom{2g}{2}-\binom{2g}{0}\right)t^2
\right)  
\]  
and so the degree $2$ coefficient of $SV_{g,1}$ is
\[
2(g+1)+ \binom{2g}{2}-1.
\]
so the overall $\dim_3(\mathcal{K}_g)$ is 
\[
\dim_3(\mathcal{K}_g)=1+2(g+1)+\binom{2g}{2}-1 - 2g-2 = \binom{2g}{2},
\]
as desired.

We move on to the inductive step. We will prove the statement for $i\le g+2$, assuming it for lower values of $i$ and all appropriate $g$. By Corollary~\ref{cor: stable value of Vg1}, 
\[
SV_{g,1}\equiv\frac{(1+t^3)(1-t^2)}{(1-t)^{2g}(1-t^2)^2}\pmod{t^i}
\] so
\begin{align*}
tSV_{g,1}&\equiv\frac{t(1+t^3)(1-t^2)}{(1-t)^{2g}(1-t^2)^2}\pmod{t^{i+1}}
\\
&\equiv
\frac{t}{(1-t)^{2g-1}}+\frac{t^2}{(1-t)^{2g+1}}\pmod{t^{i+1}}
\\&\equiv
\left(\sum_{j\geq0}\bigg[\binom{2g+j-3}{2g-2}+\binom{2g+j-2}{2g}\bigg]t^j\right) \pmod{t^{i+1}}
\end{align*}
where the last step uses the expansion of $\frac{1}{(1-t)^n}$ from Section~\ref{sec: non-orientable and open surfaces}. Thus, the coefficient of $t^i$ in this series is $\binom{2g+i-3}{2g-2}-\binom{2g+i-2}{2g}$, and combining this with the inductive premise (for $i-1$ and $g+1$) and the recurrence relation yields the result directly.
\end{proof}

Our last act will be to supply the missing proofs of Lemmas \ref{lemma:exact_sequence} and \ref{lemma:Vrecurrence}, both of which follow the same basic syntax, which we now pause to elucidate.

We fix $g\geq0$ and use the letters $q$ and $r$ to denote generic homogeneous elements of $\XX_{g+1}$, which we polarize according to the decomposition $\XX_{g+1}\cong \Lambda[a,b]\otimes\mathbb{Q}[\tilde a,\tilde b]\otimes\XX_g$ by writing

\[
\def\arraystretch{1.5}
\begin{array}{rrrrl}
q&=&&\tilde{a}^i\tilde{b}^j &q_{1,ij}\\
&+&a&\tilde{a}^{i-1}\tilde{b}^j &(q_{2,ij} - q_{3,ij})\\
&+&b&\tilde{a}^i\tilde{b}^{j-1} &(q_{2,ij } + q_{3,ij})\\
&+&ab&\tilde{a}^{i-1}\tilde{b}^{j-1} &q_{4,ij},
\end{array}\] where the $q_{kij}$ contain no factors of $a$, $b$, $\tilde a$, or $\tilde b$ and we implicitly sum over $(i,j)$. This equation defines $q_{1,ij}$ for nonnegative $i$ and $j$, $q_{2,ij}$ and $q_{3,ij}$ when at least one of $i$ and $j$ is positive, and $q_{4,ij}$ when both $i$ and $j$ are positive, and we extend to arbitrary $(i,j)$ by declaring all other cases to be zero. Note that $q_{2,i,0}=-q_{3,i,0}$ and that $q_{2,0,j}=q_{3,0,j}$.

The index $i$ counts the ``total power of $a$'' in $q$, disregarding the distinction between $a$ and $\tilde a$---likewise for $j$ and $b$---and it is evident that $\delta$ preserves both $i$ and $j$ while $\Delta$ does not. In what follows, we will apply operators $\phi_{m,n}\coloneqq\delta^m\Delta^n$ to $q$ and $r$ and impose conditions like $\phi_{m,n} q=0$ or $\phi_{m,n}q=\phi_{m',n'}r$ to obtain an infinite sequence of relations among the $q_{kij}$ and $r_{kij}$. Because $\Delta$ does not preserve $i$ and $j$, these relations will mix different values of $i$ and $j$. In order to obtain homogeneous relations, we work in the bigrading $|q_{1,ij}|=|q_{2,ij}|=(i,j)$ and $|q_{3,ij}|=|q_{4,ij}|=(i-1,j-1)$. With this bigrading in place, we will obtain a system of equations valid for each pair $(i,j)$, and so we will suppress the indices. 

As an example, we consider imposing the condition $\Delta(q)=0$, which implies that $\Delta(q_{1,ij})+q_{4,i+1,j+1}=0$ for each $(i,j)$. With our bigrading, this equation is homogeneous of degree $(i,j)$, and we abbreviate it as $$\Delta(q_1)+q_4=0.$$

There is an important special case in bidegree $(0,0)$, where $q_2$ is undefined, which we handle separately. Other potential special cases occur in bigrading $(-1,*)$ and $(*,-1)$, with $*\ge 0$, where only $q_3$ is defined. In these cases, since $q_{3,ij}=\pm q_{2,ij}$ if $i$ or $j$ is zero, the value of $q_3$ in this lowest degree is determined by $q_2$ in the next higher bidegree, and one might worry about interaction between the constraints on these two variables. Fortunately, our equations will contain pairs of the form
\begin{align*}
\phi_{m,n}q_2 &=0 \\
\phi_{m,n}q_3 &= \alpha q_4
\end{align*}
for some constant $\alpha$, so that, since $q_4=0$ in this lowest bidegree, the two constraints are identical and there is no special case.

Without further ado, we move to the proofs.

\begin{lemma}\label{lemma: K and V relationship}
If $q$ lies in $\mathcal{K}_{g+1}$, then $(q_{1,ij},-q_{3,i+1,j+1})\in \mathcal{V}(g,1)$ for each $(i,j)$, and $\delta q_{1,00}=0$. Conversely, every collection $\{q_{1,ij}, q_{3,i+1,j+1}\}$ satisfying these conditions specifies a unique element of $\mathcal{K}_{g+1}$.
\end{lemma}

\begin{proof}
The equations $\Delta q = \delta q = 0$ yield:
\begin{align*}
\Delta q_1 + q_4 &= 0\\
\Delta q_2&=0
\\
\Delta q_3&=0
\\
\Delta q_4&=0
\\
\delta q_1 + 2q_2&= 0
\\\delta q_2 & =  0
\\-\delta q_3 + q_4&=0
\\\delta q_4 & =  0.
\end{align*}
The fifth and seventh equations each have a free variables for which we will substitute, reducing the number of degrees of freedom and the number of variables. The sixth and eight equation are redundant, implied by the other equations. Then reducing, we get
\begin{align*}
\Delta q_1 + \delta q_3 &= 0\\
\Delta q_3&=0
\\\delta \Delta q_1&=0\\
\delta\Delta q_3&=0.
\end{align*}
The third and fourth equations are now redundant. Thus, in the general case, the equations are equivalent to $(q_1,-q_3)\in \mathcal{V}(g,1)$. In the special case $(i,j)=(0,0)$, we have the additional equation $\delta q_1=0$.

\end{proof}

\begin{proof}[Proof of Lemma \ref{lemma:exact_sequence}]
Lemma \ref{lemma: K and V relationship} grants that the correspondence \[q\mapsto \sum_{i,j}\tilde a^i\tilde b^j\otimes (q_{1,ij},-q_{3,i+1,j+1})\] defines a map $\mathcal{K}_{g+1}\to \mathcal{S}\otimes\mathcal{V}(g,1)$, and that this map is injective. For the righthand map, we take the composite \[\mathcal{S}\otimes \mathcal{V}(g,1)[1]\longrightarrow \mathcal{V}(g,1)[1]\xrightarrow{(q,r)\mapsto \delta q}\mathcal{K}_g,\] where the first map is induced by the augmentation of $\mathcal{S}$. Lemma \ref{lemma: K and V relationship} implies that the kernel of this map is precisely the image of $\mathcal{K}_{g+1}$.
\end{proof}

\begin{proof}[Proof of Lemma \ref{lemma:Vrecurrence}]
By definition, $(q,r)\in \mathcal{V}(g+1,n)$ if and only if $\Delta^n q = \delta\Delta^{n-1}r$ and $\Delta^n r = 0$. For $n>1$, these requirements are equivalent to the following system:
\begin{align*}
\Delta^n q_1 + n\Delta^{n-1} q_4 &= 
\delta\Delta^{n-1} r_1 + 2 \Delta^{n-1} r_2
+ (n-1)\delta \Delta^{n-2}r_4
\\
\Delta^n q_2&= -\delta\Delta^{n-1} r_2\\
\Delta^n q_3&= -\delta\Delta^{n-1} r_3 + \Delta^{n-1} r_4 \\
  \Delta^{n} q_4
&=
\delta\Delta^{n-1} r_4
\\
 \Delta^n r_1  + n\Delta^{n-1}r_4 &= 0\\
\Delta^n r_2&=0\\
\Delta^n r_3&=0\\
\Delta^{n} r_4&=0
\end{align*}
None of these equations are redundant. Rearranging, this system becomes the following:
\begin{align*}
(q_2,r_2)
&\in  \mathcal{V}(g,n)\\
(q_3+\frac{1}{n}r_1,-r_3)
&\in \mathcal{V}(g,n)\\
\Delta^{n-1}(\Delta r_1  + nr_4) &= 0\\
\Delta^{n+1} r_1&=0
\\
\Delta^n q_1 + n\Delta^{n-1} q_4 
&= 
\frac{1}{n}\delta\Delta^{n-1} r_1 + 2 \Delta^{n-1} r_2
+ \frac{n-1}{n}\delta \Delta^{n-2}(\Delta r_1 + nr_4)
\\
  \Delta^{n} q_4
&=
-\frac{1}{n}\delta\Delta^{n} r_1
\end{align*}
We rewrite the last four of these equations as
\begin{align*}
\left(nq_4 + \Delta q_1 -\frac{\delta}{n} r_1 - 2r_2,\frac{n-1}{n}(\Delta r_1+nr_4)\right)
&\in
\mathcal{V}(g,n-1)\\
\Delta^{n+1}r_1 &= 0
\\
\Delta^n\left(\Delta q_1 - \frac{\delta}{n} r_1 -2r_2\right)
&=
\delta\Delta^n r_1
\end{align*}

In summary, the original system is equivalent to the membership relations
\begin{align*}
(q_2,r_2)
&\in  \mathcal{V}(g,n)\\
\left(q_3+\frac{1}{n}r_1,-r_3\right)
&\in
\mathcal{V}(g,n)
\\
\left(nq_4 + \Delta q_1 -\frac{\delta}{n} r_1 - 2r_2,
\frac{n-1}{n}(\Delta r_1+nr_4)\right)
&\in
\mathcal{V}(g,n-1)\\
\left(q_1,\frac{n+1}{n}r_1\right)
&\in
\mathcal{V}(g,n+1)
\end{align*}
The degeneration when $(i,j)=(0,0)$ is identical except that the $(q_2,r_2)$ term does not exist. 

When $n=1$, the equations impose the relation that both sides of the third pair above, the pair that is supposed to be in $\mathcal{V}(g,n-1)$, are identically zero and the relations instead eliminate the variables $q_4$ and $r_4$. This matches our convention that $\mathcal{V}(g,0)=0$.

Examining the degrees of the pairs above yields the functional relation
\[
V_{g+1,n} = \xys{}(
V_{g,n+1}+t^2V_{g,n-1})
+ \xys{}t^3 V_{g,n} + (\xys{}-1)t^{-1}V_{g,n},
\]
and applying the identity $\frac{\xys{}-1}{\xys{}}=2t^2-t^4$ yields
\[
V_{g+1,n} = \xys{}(
V_{g,n+1}+ 2t V_{g,n} + t^2V_{g,n-1})
\]
\end{proof}

\bibliographystyle{amsalpha} 
\bibliography{references-2016.bib}
\end{document}